\documentclass[11pt]{article}
\usepackage{graphicx}
\usepackage{amsthm, amsmath, amssymb, tikz, bm}
\usepackage{mathtools}
\usepackage{xifthen}
\usepackage{comment}

\usepackage[margin=1.1in]{geometry} 

\usepackage[shortlabels]{enumitem}
\usepackage{todonotes}
\usepackage[colorlinks=true,
linkcolor=blue,citecolor=blue,
urlcolor=blue]{hyperref}

\usetikzlibrary{calc,shapes, backgrounds}
\usetikzlibrary[patterns]
\usetikzlibrary{arrows.meta}
\usetikzlibrary{decorations.markings}

\newtheorem{theorem}{Theorem}[section]
\newtheorem{lemma}[theorem]{Lemma}

\DeclarePairedDelimiter{\ceil}{\lceil}{\rceil}

\makeatletter
               {\list{}{\leftmargin=0pt 
                        \labelwidth\z@ \itemindent-\leftmargin
                        }}%
               {\endlist}
\makeatother

\newcommand{\diam}{\text{diam}}

\newcommand{\ordiam}{\overrightarrow{\text{diam}}}

\title{On the oriented diameter of near planar triangulations}

\author{
Yiwei Ge \thanks{Louisiana State University, Baton Rouge, LA, 70803
({\tt yge4@lsu.edu}).}
\and
Xiaonan Liu \thanks{Vanderbilt University, Nashville, TN, 37240
({\tt xiaonan.liu@vanderbilt.edu}).}
\and
Zhiyu Wang \thanks{Louisiana State University, Baton Rouge, LA, 70803
({\tt zhiyuw@lsu.edu}).}
}
\begin{document}

\maketitle

\begin{abstract}
In this paper, we show that the oriented diameter of any $n$-vertex $2$-connected near triangulation is at most $\ceil{\frac{n}{2}}$ (except for seven small exceptions), and the upper bound is tight. This extends a result of Wang et.al. on the oriented diameter of maximal outerplanar graphs, and improves an upper bound of $n/2+O(\sqrt{n})$ on the oriented diameter of planar triangulations by Mondal, Parthiban and Rajasingh.
\end{abstract}

\section{Introduction}\label{sec:intro}
A directed graph $D = (V(D),A(D))$ is a graph with a vertex set $V(D)$ and an edge set $A(D)$ consisting of ordered pairs of vertices, called \textit{arcs} or directed edges. We use $uv$ to denote the arc $(u,v)$, i.e., the arc oriented from $u$ to $v$. Given an undirected graph $G=(V(G),E(G))$, an \textit{orientation} of $G$ is a directed graph such that each edge in $E(G)$ is assigned an direction. Given a (directed) graph $D$ and two vertices $u,v \in V(D)$, the \textit{distance} from $u$ to $v$ in $D$, denoted by $d_D(u,v)$, is the number of edges of a shortest (directed) path from $u$ to $v$ in $D$. 
We often ignore the subscript if there is no ambiguity on the underlying graph. Given a (directed) graph $D$, the \textit{diameter} of $D$ is defined to be $\diam(D)=\max\{d_D(u,v): u,v\in V(D)\}$. 

Given an undirected graph $G$ and $v\in V(G)$, let $N_G(v)$ denote the neighborhood of $v$ in $G$, and let $N_G[v]:= N_G(v)\cup \{v\}$. An edge $e \in E(G)$ is called a \textit{bridge} if $G-e$ is disconnected. A graph is called \textit{bridgeless} if contains no bridge. 

A directed graph $D$ is called \textit{strongly connected} if for any two vertices $u,v \in V(D)$, there exists a directed path from $u$ to $v$. Robbins~\cite{Robbins1939}, showed in 1939 that every bridgeless graph has a strongly connected orientation. The \textit{oriented diameter} of a graph $G$, denoted by $\ordiam(G)$, is defined as
$$\ordiam(G)=\min\{\diam(D):\textrm{$D$ is a strongly connected orientation of $G$}\}.$$ An orientation $D$ of $G$ is called \textit{optimal} if $\ordiam(G)=\diam(D)$. Note that $\ordiam(G)\ge \diam(G)$.

Chv\'atal and Thomassen~\cite{Chvatal-Thomassen1978}, showed in 1978 that determining the oriented diameter of a given graph is NP-complete. In the same paper, Chv\'atal and Thomassen showed that every bridgeless graph $G$ with diameter $d$ satisfies $\ordiam(G)\leq 2d^2+2d$, and there exist bridgeless graphs of diameter $d$ for which every strong orientation has diameter of at least $\frac{1}{2}d^2+d$. The upper bound was recently improved by Babu, Benson, Rajendraprasad and Vaka \cite{BBRV2021} to  $1.373d^2+6.971d-1$.

The paper by Chv\'atal and Thomassen \cite{Chvatal-Thomassen1978} has led to further investigations of such bounds on the oriented diameter with respect to other graph parameters, including the diameter \cite{FMR2004,Huang-Ye2007,KLW2010}, the radius \cite{CGT1985}, the domination number \cite{FMPR2004, Laetsch-Kurz2012}, the maximum degree \cite{DGS2018}, the minimum degree \cite{Bau-Dankelmann2015,CDS2019,Surmacs2017}, the number of edges of the graph \cite{CCDS2021}, and other graph classes \cite{Chen-Chang2021,Gutin1994,GKTY2002,Gutin-Yeo2002,Huang-Ye2007, Koh-Ng2005,KRS2022, Lakshmi2011,Lakshmi-Paulraja2007, Lakshmi-Paulraja2009, Plesnik1985, Soltes1986, WCDGSV2021}. See the survey by Koh and Tay \cite{Koh-Tay2022} for more information on some of these results.

In this paper, we study the oriented diameter of planar graphs, in particular near triangulations. A \textit{near triangulation} is a plane graph such that every face except possibly the outer face is bounded by a triangle. 
 An {\it outerplanar} graph is a planar graph such that every vertex lies on the boundary of the outer face. Let $K_4^{-}$ denote the graph obtained from $K_4$ by deleting an arbitrary edge. Let $W_5$ be the wheel graph on six vertices obtained by connecting every vertex of a $C_5$ to a new vertex. Let $G_6^1, G_6^2, G_6^3, G_8^1$ be the graphs in Figure \ref{fig:exceptions}.
In~\cite{WCDGSV2021}, Wang, Chen, Dankelmann, Guo, Surmacs and Volkmann showed the following theorem on the oriented diameter of (edge-)maximal outerplanar graphs.

\begin{theorem}\cite{WCDGSV2021}\label{thm:outerplanr_diam}
    Let $G$ be an $n$-vertex maximal outerplanar graph with $n\geq 3$ that is not one of the graphs in $\{K_4^-,G_6^1, G_6^2, G_8^1\}$. Then $\ordiam(G)\leq \ceil{\frac{n}{2}}$ and this upper bound is tight.
\end{theorem}

   \begin{figure}[htb]
    \hbox to \hsize{
	\hfil
	\resizebox{4cm}{!}{\begin{tikzpicture}[scale=1, Wvertex/.style={circle, draw=black, fill=white, scale=3}, bvertex/.style={circle, draw=black, fill=black, scale=0.3},rvertex/.style={circle, draw=red, fill=red, scale=0.2}]

\node [bvertex, label={[font=\small] left:$v_1$}] (v1) at (-1,0) {};
\node [bvertex, label={[font=\small] above:$v_2$}] (v2) at (0,1) {};
\node [bvertex, label={[font=\small] above:$v_3$}] (v3) at (1,1) {};
\node [bvertex, label={[font=\small] right:$v_4$}] (v4) at (2,0) {};
\node [bvertex, label={[font=\small] below:$v_5$}] (v5) at (1,-1) {};
\node [bvertex, label={[font=\small] below:$v_6$}] (v6) at (0,-1) {};

\draw (v1) -- (v2);
\draw (v2) -- (v3);
\draw (v3) -- (v4);
\draw (v4) -- (v5);
\draw (v5) -- (v6);
\draw (v6) -- (v1);
\draw (v2) -- (v6);
\draw (v3) -- (v6);
\draw (v3) -- (v5);

\end{tikzpicture}	}%
	\hfil
	\resizebox{3.5cm}{!}{\begin{tikzpicture}[scale=1, Wvertex/.style={circle, draw=black, fill=white, scale=3}, bvertex/.style={circle, draw=black, fill=black, scale=0.3},rvertex/.style={circle, draw=red, fill=red, scale=0.2}]

\node [bvertex, label={[font=\small] left:$v_1$}] (v1) at (180:1) {};
\node [bvertex, label={[font=\small] above:$v_2$}] (v2) at (120:1) {};
\node [bvertex, label={[font=\small] above:$v_3$}] (v3) at (60:1) {};
\node [bvertex, label={[font=\small] right:$v_4$}] (v4) at (0:1) {};
\node [bvertex, label={[font=\small] below:$v_5$}] (v5) at (300:1) {};
\node [bvertex, label={[font=\small] below:$v_6$}] (v6) at (240:1) {};

\draw (v1) -- (v2);
\draw (v2) -- (v3);
\draw (v3) -- (v4);
\draw (v4) -- (v5);
\draw (v5) -- (v6);
\draw (v6) -- (v1);
\draw (v2) -- (v6);
\draw (v3) -- (v6);
\draw (v4) -- (v6);

\end{tikzpicture}	}%
	\hfil
    \resizebox{4.7cm}{!}{\begin{tikzpicture}[scale=1, Wvertex/.style={circle, draw=black, fill=white, scale=3}, bvertex/.style={circle, draw=black, fill=black, scale=0.3},rvertex/.style={circle, draw=red, fill=red, scale=0.2}]

\node [bvertex, label={[font=\small] left:$v_1$}] (v1) at (150:1) {};
\node [bvertex, label={[font=\small] above:$v_2$}] (v2) at (0,0) {};
\node [bvertex, label={[font=\small] above:$v_3$}] (v3) at (30:1) {};
\node [bvertex, label={[font=\small] right:$v_4$}] (v4) at (2,-0.25) {};
\node [bvertex, label={[font=\small] below:$v_5$}] (v5) at (0.866,-1) {};
\node [bvertex, label={[font=\small] below:$v_6$}] (v6) at (-90:1) {};

\draw (v1) -- (v2);
\draw (v2) -- (v3);
\draw (v3) -- (v4);
\draw (v4) -- (v5);
\draw (v5) -- (v6);
\draw (v6) -- (v1);
\draw (v2) -- (v6);
\draw (v3) -- (v6);
\draw (v3) -- (v5);
\draw (v1) -- (v3);

\end{tikzpicture}	}%
    \hfil
    \resizebox{3.5cm}{!}{\begin{tikzpicture}[scale=1, Wvertex/.style={circle, draw=black, fill=white, scale=3}, bvertex/.style={circle, draw=black, fill=black, scale=0.3},rvertex/.style={circle, draw=red, fill=red, scale=0.2}]

\node [bvertex, label={[font=\small] above:$v_1$}] (v1) at (-1,1) {};
\node [bvertex, label={[font=\small] above:$v_2$}] (v2) at (0,1) {};
\node [bvertex, label={[font=\small] above:$v_3$}] (v3) at (1,1) {};
\node [bvertex, label={[font=\small] above:$v_4$}] (v4) at (2,1) {};
\node [bvertex, label={[font=\small] below:$v_5$}] (v5) at (2,-1) {};
\node [bvertex, label={[font=\small] below:$v_6$}] (v6) at (1,-1) {};
\node [bvertex, label={[font=\small] below:$v_7$}] (v7) at (0,-1) {};
\node [bvertex, label={[font=\small] below:$v_8$}] (v8) at (-1,-1) {};

\draw (v1) -- (v2);
\draw (v2) -- (v3);
\draw (v3) -- (v4);
\draw (v4) -- (v5);
\draw (v5) -- (v6);
\draw (v6) -- (v7);
\draw (v7) -- (v8);
\draw (v8) -- (v1);
\draw (v2) -- (v7);
\draw (v3) -- (v6);
\draw (v1) -- (v7);
\draw (v3) -- (v7);
\draw (v3) -- (v5);

\end{tikzpicture}	}%
    \hfil
    }
    \caption{Some small exceptions to Theorem \ref{thm:near_triangulation_diameter}: $G_6^1, G_6^2, G_6^3, G_8^1$ (from left to right).}
    \label{fig:exceptions}
    \end{figure}
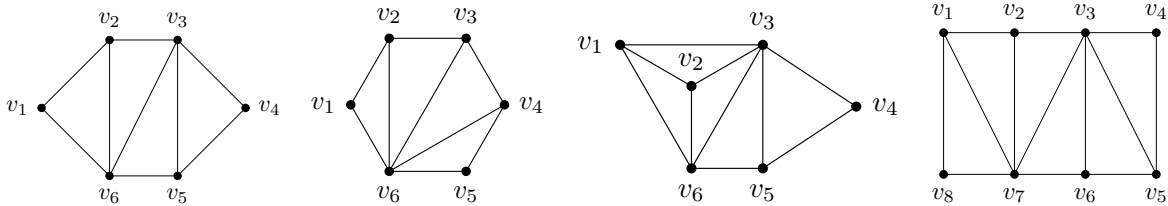

Recently, Mondal, Parthiban and Rajasingh \cite{MPR2023} studied the oriented diameter of planar triangulations, which are the maximal planar graphs. They showed the following theorem.

\begin{theorem}\cite{MPR2023}\label{thm:planar_diam}
Let $G$ be an $n$-vertex planar triangulation. Then $\ordiam(G) \leq \frac{n}{2}+O(\sqrt{n})$. Moreover, for $n$ that is a multiple of $3$, there exist $n$-vertex planar triangulations with oriented diameter at least $n/3$. 
\end{theorem}

They \cite{MPR2023} asked whether the upper bound could be improved. In this paper, we determine a tight upper bound (except for seven small exceptions) on the oriented diameter of a $2$-connected near triangulation. 

\begin{theorem}\label{thm:near_triangulation_diameter}
    Let $G$ be an $n$-vertex $2$-connected near triangulation that is not one of the graphs in $\{K_4^-, K_4, W_5, G_6^1, G_6^2, G_6^3, G_8^1\}$. Then 
    $\ordiam(G)\leq \ceil*{\frac{n}{2}}.$
\end{theorem}

Note that both maximal outerplanar graphs and planar triangulations are $2$-connected near triangulations. Therefore, Theorem \ref{thm:near_triangulation_diameter} extends Theorem \ref{thm:outerplanr_diam} and improves the upper bound in Theorem \ref{thm:planar_diam}. Moreover, by Theorem \ref{thm:outerplanr_diam}, there exist $n$-vertex maximal outerplanar graphs (which are near triangulations) that have oriented diameter $\ceil{n/2}$ for every integer $n$ with $n\ge 5$. Hence the upper bound in Theorem \ref{thm:near_triangulation_diameter} is tight.\\

{\bf Organization and terminology.} In Section \ref{sec:preliminaries}, we show some lemmas about the structure of near triangulations with certain outer cycle. In Section \ref{sec:main-theorem}, we give a proof of Theorem \ref{thm:near_triangulation_diameter}. For convenience, for the rest of the paper, we assume that a near triangulation is $2$-connected, i.e., its outer face is bounded by a cycle.
We conclude this section with some terminology and notation. For any positive integer $k$, let $[k]:=\{ 1, 2, \ldots, k\}$. Let $G$ be a plane graph. The {\it outer walk} of $G$ consists of vertices and edges of $G$ incident with the outer face of $G$. If the outer walk is a cycle in $G$, we call it  {\it outer cycle} instead. 
For a cycle $C$ in  $G$, we use $\overline{C}$ to denote the subgraph of $G$ consisting of all vertices and edges of $G$ contained in the closed disc in the plane bounded by $C$. The {\it interior} of $C$ is then defined as the subgraph $\overline{C}-C$. For any distinct vertices $u,v\in V(C)$, we use $uCv$ to denote the subpath of $C$ from $u$ to $v$ in clockwise order. 

\section{Preliminaries}\label{sec:preliminaries}

The following lemma is clear, we state it below without proof.
\begin{lemma}\label{lem:spanning_subgraph}
Let $G$ be a graph and $F$ be a spanning subgraph of $G$. Then $\ordiam(G)\le \ordiam(F)$.
\end{lemma}
In \cite{WCDGSV2021}, Wang et.al.~applied induction on the number of vertices to show Theorem~\ref{thm:outerplanr_diam} and they observed three structures that help reduce a maximal outerplanar graph to a graph with fewer vertices. Observe that those structures can also help reduce a near triangulation. Note that every degree-$2$ vertex $v$ in a near triangulation is contained in its outer cycle $C$ and $N(v)\subseteq V(C)$.
\begin{lemma}\cite{WCDGSV2021}\label{lem:two_deg2}
Let $G$ be a near triangulation and $v_1, v_2$ be two distinct degree-$2$ vertices in $G$. Let $H=G-\{v_1, v_2\}$. If $N(v_1)\cap N(v_2)\ne \emptyset$, then $\ordiam(G)\le \max \{ \ordiam (H) +1, 4\}$.
\end{lemma}

\begin{lemma}\cite{WCDGSV2021}\label{lem:four_deg2}
Let $G$ be a near triangulation and $v_1, v_2, v_3, v_4$ be four distinct degree-$2$ vertices in $G$. Let $H=G-\{v_1, v_2, v_3, v_4\}$. Then $\ordiam(G)\le \ordiam (H) +2$.
\end{lemma}

\begin{lemma}\cite{WCDGSV2021}\label{lem:three_deg2}
Let $G$ be a near triangulation and $v_1, v_2, v_3$ be three distinct degree-$2$ vertices in $G$ such that for each $i\in [3]$, $v_i$ has a neighbor $v_i'$ of degree three in $G$. Let $H=G-\{v_1, v_2, v_3, v_1', v_2', v_3'\}$. Then $\ordiam(G)\le \ordiam (H) +3$.
\end{lemma}

The next lemma follows from Theorem \ref{thm:outerplanr_diam}.

\begin{lemma}\label{lem:outerplanar_vtx}
Let $G$ be an $n$-vertex maximal outerplanar graph with $n\geq 3$ and $v\in V(G)$. Then there exists an optimal orientation $D$ of $G$ such that $\diam(D)=\ordiam(G)\le {\frac{n}{2}}+1$ and $\max\{d_D(u,v),d_D(v,u)\}\leq\ceil*{\frac{n}{2}}$ for all $u\in V(G)$.
\end{lemma}

\begin{proof}
    By Theorem \ref{thm:outerplanr_diam}, it suffices to verify the lemma when $G\in\{K_4^-, G_6^1,G_6^2,G_8^1\}$. For these four exception graphs, we give explicit optimal orientations (see Figure~\ref{fig:exceptions_orientation}). In each orientation in Figure~\ref{fig:exceptions_orientation}, the vertices $v$ for which Lemma \ref{lem:outerplanar_vtx} holds are colored red. Each vertex in $G$ is isomorphic to some red vertex. This proves Lemma~\ref{lem:outerplanar_vtx}.
\end{proof}

  \begin{figure}[htb]
    \hbox to \hsize{
    \resizebox{10cm}{!}{\input{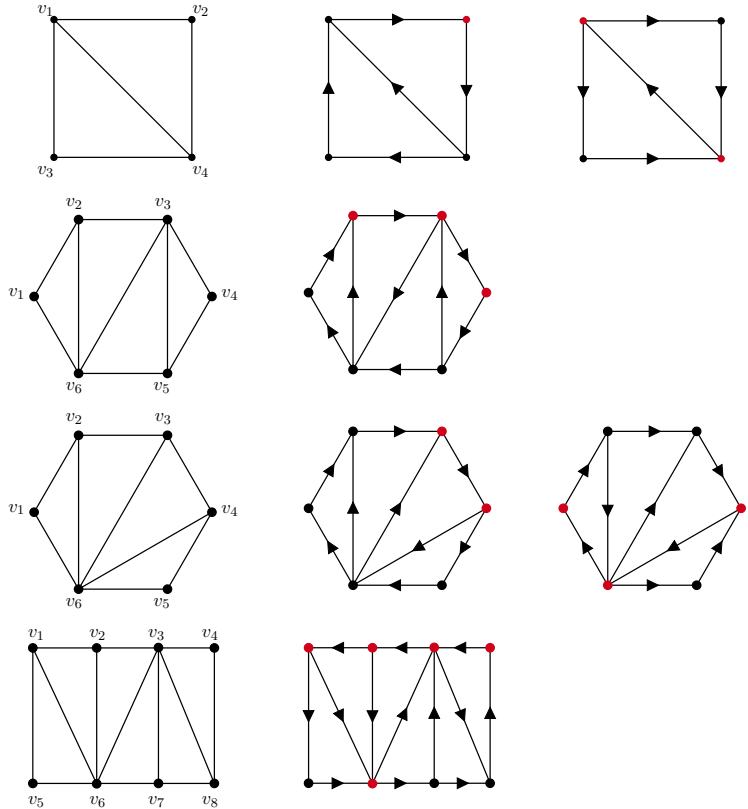}}
    }
    \caption{Explicit optimal orientations to exceptions in Theorem \ref{thm:outerplanr_diam}.}
    \label{fig:exceptions_orientation}
    \end{figure}
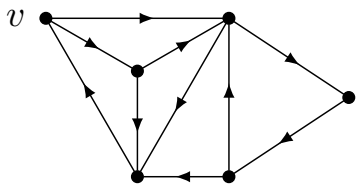
\begin{figure}[htb]
    \hbox to \hsize{
	\hfil
	\resizebox{5cm}{!}{\begin{tikzpicture}[scale=1, Wvertex/.style={circle, draw=black, fill=white, scale=3}, bvertex/.style={circle, draw=black, fill=black, scale=0.3},rvertex/.style={circle, draw=red, fill=red, scale=0.2}, decoration={markings, 
	mark= at position 0.6 with {\arrow{latex}}}]

\node [bvertex,  label={[font=\small] left:$v$}] (v1) at (150:1) {};
\node [bvertex] (v2) at (0,0) {};
\node [bvertex] (v3) at (30:1) {};
\node [bvertex] (v4) at (2,-0.25) {};
\node [bvertex] (v5) at (0.866,-1) {};
\node [bvertex] (v6) at (-90:1) {};

\draw [postaction={decorate}] (v1) -- (v2);
\draw [postaction={decorate}] (v2) -- (v3);
\draw [postaction={decorate}] (v3) -- (v4);
\draw [postaction={decorate}] (v4) -- (v5);
\draw [postaction={decorate}] (v5) -- (v6);
\draw [postaction={decorate}] (v6) -- (v1);
\draw [postaction={decorate}] (v2) -- (v6);
\draw [postaction={decorate}] (v3) -- (v6);
\draw [postaction={decorate}] (v5) -- (v3);
\draw [postaction={decorate}] (v1) -- (v3);

\end{tikzpicture}	}%
	\hfil
    }
    \caption{An optimal orientation of $G_6^3$.}
    \label{fig:optimalorientation}
    \end{figure}
\begin{lemma}\label{lem:planartriangulation_vtx}
Let $G$ be an $n$-vertex near triangulation in $\{K_4, G_6^3\}$ with outer cycle $C$ and let $v$ be a vertex in $C$ such that $d(v)=3$ and $v$ has a neighbor contained in $V(G)\backslash V(C)$. Then there exists an optimal orientation $D$ of $G$ such that $\diam(D)=\ordiam(G)=\frac{n}{2}+1$ and $\max\{d_D(u,v),d_D(v,u)\}\leq\frac{n}{2}$ for all $u\in V(G)$.
\end{lemma}

\begin{proof} Since we give such orientation for $K_4^-$ in Figure~\ref{fig:exceptions_orientation} and $K_4^-$ is a spanning subgraph of $K_4$, we only need to give an orientation of $G_6^3$ satisfying the condition. We show this in Figure~\ref{fig:optimalorientation}. Note that there is a unique vertex $v$ in $G_6^3$ satisfying the lemma. This proves Lemma~\ref{lem:planartriangulation_vtx}.
\end{proof}

We also show the following lemma, which describes some structural properties of certain near triangulations.

\begin{lemma}\label{lem:deg2vtx}
Let $G$ be a near triangulation with outer cycle $C$ such that $V(G)\backslash V(C) \ne \emptyset$ and there is no vertex $v\in V(G)\backslash V(C)$ such that $v, u,w $ form a facial triangle with some $e=uw\in E(C)$.  Let $S:=\{ v\in V(C): \text{ $v$ has a neighbor in $V(G)\backslash V(C)$} \}$. Then the following statements hold.
\begin{itemize}
\item [(i)] $|S|\ge 3$.
\item [(ii)] $G$ has at least three vertices of degree two contained in $V(C)$. If $G$ has exactly three degree two vertices contained in $V(C)$, then $|S|=3$ and $T:= G[S]$ is a triangle such that all vertices in $V(G)\backslash V(C)$ are contained in the interior of $\overline{T}$.
\end{itemize}
\begin{proof} Since $G$ is a near triangulation with outer cycle $C$ and  $V(G)\backslash V(C) \ne \emptyset$, we know that $S\ne \emptyset$. It follows that $|S|\ge 3$ as $G$ is simple. We apply induction on $|V(G)|$ to show (ii). Assume that (ii) holds for all such near triangulations on smaller than $|V(G)|$ vertices.

Let $C=v_1v_2\ldots v_k v_1$ and $S=\{u_1, u_2, \ldots, u_p\}$ for some integer $p\ge 3$ such that $v_1, v_2, \ldots, v_k$ appear on $C$ in clockwise order and $u_1, u_2, \ldots, u_p$ appear on $C$ in clockwise order. For convenience, we assume $v_{k+1}=v_1$ and $u_{p+1}=u_1$. Suppose $u_iu_{i+1}\in E(C)$ for some $i$ with $1\le i \le p$. We may assume that $u_1u_2\in E(C)$. Note that there is no vertex $v\in V(G)\backslash V(C)$ such that $v$ is adjacent to $u_1$ and $u_2$ by our assumptions on $G$. Since $G$ is a near triangulation with outer cycle $C$, $u_1$ and $u_2$ must have a common neighbor $w$ that lies on $C$ such that $u_1u_2w u_1$ is a facial triangle and $u_1w, u_2w\notin E(C)$ as $u_i\in S$ for each $i\in [2]$. Consider the cycles $C_1=v C u_1 u_2 w$ and $C_2=u_2 C w u_1 u_2$. 
Observe that since $u_1, u_2\in S$, we have that for each $i=1,2$, $\overline{C_i}$ is a near triangulation with outer cycle $C_i$ such that $V(\overline{C_i})\backslash V(C_i) \ne \emptyset$ and for any $v\in V(\overline{C_i})\backslash V(C_i)$, $v$ and $e$ don't form a facial triangle in $\overline{C_i}$ for every $e\in E(C_i)$. By induction, $\overline{C_i}$ has at least three vertices of degree two contained in $C_i$. We know that $d_{\overline{C_i}} (u_{3-i})=2, d_{\overline{C_i}} (u_i)\ge 3$, and $d_{\overline{C_i}} (w)\ge 3$. It follows that $G$ has at least two degree-$2$ vertices contained in $V(wCu_1)$ and at least two degree-$2$ vertices contained in $V(u_2 C w)$. Hence $G$ has at least four vertices of degree two contained in $V(C)$. Now we assume that $u_iu_{i+1}\notin E(C)$ for all $i\in [p]$. Suppose $u_iu_{i+1}\notin E(G)$ for some $i\in [p]$, i.e., $u_i$ and $u_{i+1}$ are not adjacent in $G$. Without loss of generality, we may assume $u_1u_2\notin E(G)$, $u_1=v_1, u_2=v_t$ for some integer $t\ge 3$. 
It follows that $v_2\notin S$ and $v_{t-1}\notin S$. Let $w_1$ be the common neighbor of $u_1$ and $v_2$ lying on $C$, and $w_2$ be the common neighbor of $u_2$ and $v_{t-1}$ such that $u_1v_2w_1u_1, u_2v_{t-1}w_2u_2$ are facial triangles in $G$. Observe that $w_1\ne u_1$ and $w_2\ne u_2$ as $u_1u_2\notin E(G)$. If $w_1$ is an internal vertex of $u_2 C u_1$ then let $D_1:=w_1C u_1v_2 w_1$ and $D_2:= u_1 v_2 C w_1 u_1$.
Note that $u_1 v_2 w_1 u_1$ is a facial cycle, $u_1, u_2\in S$ and $v_2\notin S$. It follows that for each $i\in [2]$, $\overline{D_i}$ is a near triangulation with outer cycle $D_i$ such that $V(\overline{D_i})\backslash V(D_i) \ne \emptyset$ and there is no $v\in V(\overline{D_i})\backslash V(D_i)$ such that $v$ and $e$ form a facial triangle in $\overline{D_i}$ for some $e\in E(D_i)$. We apply induction to $\overline{D_1}, \overline{D_2}$, and similar to before, we obtain at least two vertices of degree two contained in $w_1C u_1$ and two other vertices of degree two contained in $v_2 C w_1$. Hence we obtain at least four vertices of degree two contained in $C$. Similarly, we obtain at least four vertices of degree two in $C$ if $w_2$ is an internal vertex of $u_2Cu_1$. Thus we can assume that both $w_1$ and $w_2$ are contained in $v_2 C v_{t-1}$. Now let $w_1'$ be the neighbor of $u_1$ in $v_2 C v_{t-1}$ with the maximum index (i.e., furthest away from $u_1$ in $C$). Note that $w_1'\notin S$. Hence the common neighbor $z$ of $u_1$ and $w_1'$ that forms a facial triangle with $u_1w_1'$ must lie on $u_2 C u_1$ and $z \notin \{u_2, v_k\}$. Let $D_1:= z C w_1' z$. Applying induction to $\overline{D_1}$, we obtain that there are at least three vertices of degree two which are internal vertices of $z C w_1'$ (as $d_{\overline{D_1}}(w_1')\geq 3$ and $d_{\overline{D_1}}(z) \geq 3$). Moreover, observe that $w_2 C u_2 w_2$ is a maximal outerplanar graph. Thus there exists at least one vertex of degree two which is an internal vertex of $w_2 C u_2$. Hence we obtain at least four vertices of degree two in $C$.

Now we may assume that $u_iu_{i+1}\in E(G)\backslash E(C)$ for all $i \in [j]$. We know that every internal vertex in $u_i C u_{i+1}$ has no neighbor in $V(G)\backslash V(C)$. Let $C_i'=u_iC u_{i+1} u_i$. Then $\overline{C_i'}$ is a maximal outerplanar graph with outer cycle $C_i'$. Recall that every outerplanar graph has at least two vertices of degree $2$ (since the dual of an outerplanar graph is a forest).
Hence there are at least two vertices of degree two in $\overline{C_i'}$. Note that either $d_{\overline{C_i'}(u_i)}\ge 3$ or $d_{\overline{C_i'}(u_{i+1})}\ge 3$. It follows that $G$ has at least one vertex of degree two contained in $V(u_i C u_{i+1})$. Since $k\ge 3$, we know $G$ has at least three vertices of degree two in $V(C)$. If $G$ has exactly three vertices of degree two in $V(C)$, then it follows from the above argument that  $S=\{u_1, u_2, u_3\}$, $u_iu_{i+1}\in E(G)\backslash E(C)$ and $T=G[S]=u_1u_2u_3u_1$ is a triangle in $G$ such that $V(G)\backslash V(C) \subseteq V(\overline{T}-T)$.
\end{proof}
\end{lemma}

\section{Proof of Theorem \ref{thm:near_triangulation_diameter}}\label{sec:main-theorem}

We will show Theorem \ref{thm:near_triangulation_diameter} by induction on the number of vertices in $G$. Theorem \ref{thm:near_triangulation_diameter} for near triangulations on $n\leq 8$ vertices are confirmed by computer searches\footnote{See the code in \url{https://github.com/wzy3210/Oriented_diameter}.}. So we assume that $n\geq 9$.  Let $G_0$ be an $n$-vertex near triangulation with $n\ge 9$. Note that $G$ is not one of the graphs in $\{K_4^-, K_4, W_5, G_6^1, G_6^2, G_6^3, G_8^1\}$. Let $C_0$ be the outer cycle of $G_0$. We remove an edge $e=uw\in E(C_0)$ from $C_0$ if there exists a vertex $v\in V(G_0)\backslash V(C_0)$ such that $v, u, w$ form a facial triangle. Note that $G_0-e$ is an $n$-vertex $2$-connected near triangulation with the outer cycle $(C_0-e) \cup uvw$. We perform the same process on $G-e$ and repeat this process until we get an $n$-vertex $2$-connected near triangulation $G$ with outer cycle $C$ such that there is no vertex $v$ in $V(G)\backslash V(C)$ that forms a facial triangle with some $e\in E(C)$. By Lemma \ref{lem:spanning_subgraph}, $\ordiam(G_0) \leq \ordiam(G)$ since $G$ is a spanning subgraph of $G_0$. Hence it suffices to show that $\ordiam(G)\leq \ceil*{\frac{n}{2}}$. 

Suppose $V(G)\backslash V(C)=\emptyset$. Then $G$ is a maximal outerplanar graph on at least $9$ vertices, and we have $\ordiam(G)\le \ceil*{\frac{n}{2}}$ by Theorem \ref{thm:outerplanr_diam}. Hence we may assume that $V(G)\backslash V(C)\ne \emptyset$. Let $A$ be the set of vertices of degree two in $G$ (which all lie on $C$). It follows from Lemma~\ref{lem:deg2vtx} that $|A|\ge 3$. If $|A|\ge 4$, by Lemmas~\ref{lem:four_deg2}, we apply induction on $G-A'$, where $A'\subseteq A$ and $|A'|=4$. Note that $G-A'$ is a smaller near triangulation on at least $5$ vertices such that $G-A'$ is not outerplanar and hence $\ordiam(G)\le \ordiam(G-A')+2\le \ceil*{\frac{n}{2}}$ unless $G-A'\in\{W_5, G_6^3\}$. If $G-A'\in\{W_5,G_6^3\}$,  we know that $n=10$ and the outer cycle of $G-A'$ has length five. Since $|A'|=4$, there exist two degree-$2$ vertices $u,v$ in $G$ such that $N(u)\cap N(v)\ne \emptyset$. Let $H=G-\{u, v\}$. Then it follows from Lemma~\ref{lem:two_deg2} that $\ordiam(G)\le \max\{ \ordiam(H)+1, 4\}$. Since $H=G-\{u,v\}$ is a near triangulation on $8$ vertices and $H$ is not outerplanar, $\ordiam(H)\le \ceil*{\frac{8}{2}}=4$. Therefore, $\ordiam(G)\le 5$.

Now we assume that $|A|=3$. Let $S=\{ v\in V(C): \text{ $v$ has a neighbor in $V(G)\backslash V(C)$} \}$.  By Lemma~\ref{lem:deg2vtx}, $|S|=3$ and $G[S]$ is a triangle such that all vertices in $V(G)\backslash V(C)$ are contained in the interior of $G[S]$. Let $T=G[S]=u_1u_2u_3u_1$ such that $u_1,u_2,u_3$ appear on $C$ in clockwise order. Note that $\overline{T}$ is a planar triangulation on at least $4$ vertices. For convenience, let $u_4=u_1$. Let $C_i:= u_i C u_{i+1} u_i $ for each $i\in [3]$. We know that $O_i:=\overline{C_i}$ is a maximal outerplanar graph with $|V(O_i)|\ge 3$. Since $G$ has exactly three degree-$2$ vertices, there is exactly one degree-$2$ vertex, say $v_i$, in $V(u_i C u_{i+1})$ for every $i\in [3]$. Hence $A=\{v_1, v_2, v_3\}$.

\begin{figure}[htb]
    \hbox to \hsize{
	\hfil
	\resizebox{4cm}{!}{
 
\tikzset{
pattern size/.store in=\mcSize, 
pattern size = 5pt,
pattern thickness/.store in=\mcThickness, 
pattern thickness = 0.3pt,
pattern radius/.store in=\mcRadius, 
pattern radius = 1pt}
\makeatletter
\pgfutil@ifundefined{pgf@pattern@name@_il56rc5sm lines}{
\pgfdeclarepatternformonly[\mcThickness,\mcSize]{_il56rc5sm}
{\pgfqpoint{0pt}{0pt}}
{\pgfpoint{\mcSize+\mcThickness}{\mcSize+\mcThickness}}
{\pgfpoint{\mcSize}{\mcSize}}
{\pgfsetcolor{\tikz@pattern@color}
\pgfsetlinewidth{\mcThickness}
\pgfpathmoveto{\pgfpointorigin}
\pgfpathlineto{\pgfpoint{\mcSize}{0}}
\pgfusepath{stroke}}}
\makeatother

 
\tikzset{
pattern size/.store in=\mcSize, 
pattern size = 5pt,
pattern thickness/.store in=\mcThickness, 
pattern thickness = 0.3pt,
pattern radius/.store in=\mcRadius, 
pattern radius = 1pt}
\makeatletter
\pgfutil@ifundefined{pgf@pattern@name@_j0dm7k0xv lines}{
\pgfdeclarepatternformonly[\mcThickness,\mcSize]{_j0dm7k0xv}
{\pgfqpoint{0pt}{0pt}}
{\pgfpoint{\mcSize+\mcThickness}{\mcSize+\mcThickness}}
{\pgfpoint{\mcSize}{\mcSize}}
{\pgfsetcolor{\tikz@pattern@color}
\pgfsetlinewidth{\mcThickness}
\pgfpathmoveto{\pgfpointorigin}
\pgfpathlineto{\pgfpoint{\mcSize}{0}}
\pgfusepath{stroke}}}
\makeatother

 
\tikzset{
pattern size/.store in=\mcSize, 
pattern size = 5pt,
pattern thickness/.store in=\mcThickness, 
pattern thickness = 0.3pt,
pattern radius/.store in=\mcRadius, 
pattern radius = 1pt}
\makeatletter
\pgfutil@ifundefined{pgf@pattern@name@_shiml2p6c lines}{
\pgfdeclarepatternformonly[\mcThickness,\mcSize]{_shiml2p6c}
{\pgfqpoint{0pt}{0pt}}
{\pgfpoint{\mcSize+\mcThickness}{\mcSize+\mcThickness}}
{\pgfpoint{\mcSize}{\mcSize}}
{\pgfsetcolor{\tikz@pattern@color}
\pgfsetlinewidth{\mcThickness}
\pgfpathmoveto{\pgfpointorigin}
\pgfpathlineto{\pgfpoint{\mcSize}{0}}
\pgfusepath{stroke}}}
\makeatother
\tikzset{every picture/.style={line width=0.75pt}} 

\begin{tikzpicture}[x=0.75pt,y=0.75pt,yscale=-1,xscale=1]

\draw  [color={rgb, 255:red, 255; green, 255; blue, 255 }  ,draw opacity=1 ][fill={rgb, 255:red, 155; green, 155; blue, 155 }  ,fill opacity=0.3 ] (220.18,83.4) -- (306.37,220.93) -- (133.99,220.93) -- cycle ;
\draw    (219.93,83.4) -- (133.99,220.93) ;
\draw    (305.87,220.93) -- (133.99,220.93) ;
\draw    (219.93,83.4) -- (305.87,220.93) ;
\draw  [draw opacity=0][pattern=_il56rc5sm,pattern size=15pt,pattern thickness=0.75pt,pattern radius=0pt, pattern color={rgb, 255:red, 0; green, 0; blue, 0}] (133.98,220.93) .. controls (111.68,207.07) and (111.3,165.39) .. (133.34,127.08) .. controls (155.66,88.29) and (192.52,67.63) .. (215.67,80.94) .. controls (218,82.29) and (220.1,83.92) .. (221.97,85.81) -- (175.25,151.19) -- cycle ; \draw   (133.98,220.93) .. controls (111.68,207.07) and (111.3,165.39) .. (133.34,127.08) .. controls (155.66,88.29) and (192.52,67.63) .. (215.67,80.94) .. controls (218,82.29) and (220.1,83.92) .. (221.97,85.81) ;  
\draw  [draw opacity=0][pattern=_j0dm7k0xv,pattern size=15pt,pattern thickness=0.75pt,pattern radius=0pt, pattern color={rgb, 255:red, 0; green, 0; blue, 0}] (306.37,220.93) .. controls (328.67,207.07) and (329.05,165.4) .. (307.01,127.09) .. controls (284.7,88.29) and (247.84,67.64) .. (224.69,80.95) .. controls (222.36,82.29) and (220.26,83.93) .. (218.38,85.82) -- (265.1,151.2) -- cycle ; \draw   (306.37,220.93) .. controls (328.67,207.07) and (329.05,165.4) .. (307.01,127.09) .. controls (284.7,88.29) and (247.84,67.64) .. (224.69,80.95) .. controls (222.36,82.29) and (220.26,83.93) .. (218.38,85.82) ;  
\draw  [draw opacity=0][pattern=_shiml2p6c,pattern size=15pt,pattern thickness=0.75pt,pattern radius=0pt, pattern color={rgb, 255:red, 0; green, 0; blue, 0}] (136.78,222.13) .. controls (136.74,222.66) and (136.73,223.19) .. (136.73,223.72) .. controls (136.73,247.07) and (174.45,266) .. (220.99,266) .. controls (267.52,266) and (305.25,247.07) .. (305.25,223.72) .. controls (305.25,220.93) and (304.71,218.2) .. (303.68,215.56) -- (220.99,223.72) -- cycle ; \draw   (136.78,222.13) .. controls (136.74,222.66) and (136.73,223.19) .. (136.73,223.72) .. controls (136.73,247.07) and (174.45,266) .. (220.99,266) .. controls (267.52,266) and (305.25,247.07) .. (305.25,223.72) .. controls (305.25,220.93) and (304.71,218.2) .. (303.68,215.56) ;  
\draw  [fill={rgb, 255:red, 0; green, 0; blue, 0 }  ,fill opacity=1 ] (213.91,83.4) .. controls (213.91,86.72) and (216.6,89.42) .. (219.93,89.42) .. controls (223.26,89.42) and (225.95,86.72) .. (225.95,83.4) .. controls (225.95,80.07) and (223.26,77.38) .. (219.93,77.38) .. controls (216.6,77.38) and (213.91,80.07) .. (213.91,83.4) -- cycle ;
\draw  [fill={rgb, 255:red, 0; green, 0; blue, 0 }  ,fill opacity=1 ] (130.76,222.13) .. controls (130.76,225.46) and (133.46,228.16) .. (136.78,228.16) .. controls (140.11,228.16) and (142.81,225.46) .. (142.81,222.13) .. controls (142.81,218.81) and (140.11,216.11) .. (136.78,216.11) .. controls (133.46,216.11) and (130.76,218.81) .. (130.76,222.13) -- cycle ;
\draw  [fill={rgb, 255:red, 0; green, 0; blue, 0 }  ,fill opacity=1 ] (297.65,221.58) .. controls (297.65,224.9) and (300.35,227.6) .. (303.68,227.6) .. controls (307,227.6) and (309.7,224.9) .. (309.7,221.58) .. controls (309.7,218.25) and (307,215.56) .. (303.68,215.56) .. controls (300.35,215.56) and (297.65,218.25) .. (297.65,221.58) -- cycle ;

\draw (141.2,139.48) node [anchor=north west][inner sep=0.75pt]    {$O_3$};
\draw (276.49,137.48) node [anchor=north west][inner sep=0.75pt]    {$O_1$};
\draw (207.84,239.65) node [anchor=north west][inner sep=0.75pt]    {$O_2$};
\draw (200.84,150.51) node [anchor=north west][inner sep=0.75pt]     {$\overline{T}$};
\draw (213.34,58.4) node [anchor=north west][inner sep=0.75pt]    {$u_1$};
\draw (308.37,226.37) node [anchor=north west][inner sep=0.75pt]    {$u_2$};
\draw (121,226.37) node [anchor=north west][inner sep=0.75pt]    {$u_3$};

\end{tikzpicture}}%
	\hfil
    }
    \caption{$G= \overline{T} \cup O_1 \cup O_2 \cup O_3$.}
    \label{fig:G_decomposition}
    \end{figure}
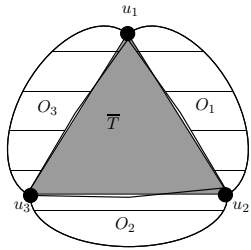

Case 1: Suppose there are at least two $O_i$ with $|V(O_i)|=3$. Without loss of generality, we may assume $|V(O_1)|=|V(O_2)|=3$. Then $v_1$ and $v_2$ have a common neighbor $u_2$. Let $H=G-\{v_1, v_2\}$. Note that $H$ is a near triangulation on  $n-2$ vertices (with $n-2\ge 7$) and $H$ is not outerplanar. This implies that $H$ is not one of the exception graphs, and hence, $\ordiam(H)\le \ceil* {\frac{n-2}{2}}$ by induction. It follows from Lemma~\ref{lem:two_deg2} that $\ordiam(G)\le \max \{ \ordiam(H)+1, 4\}\le \ceil*{\frac{n}{2}}$. 

\vspace{2mm}

Case 2: Suppose there is exactly one $O_i$ with $|V(O_i)|=3$ and we may assume that $|V(O_3)|=3$, i.e., $O_3=v_3 u_1 u_3 v_3$. Let $n_i:=|V(O_i)|$ for each $i\in [2]$ and let $n_T: =|V(\overline{T})|$. Then $n_i\ge 4$ for $i\in [2]$ and $n_T\ge 4$. Since $n_1+n_2+n_T=n+3$, we have $n_i\le n-5$ for $i\in [2]$ and $n_T\le n-5$. First we orient the edges in $\overline{T}$. If $\overline{T}=K_4$, Lemma \ref{lem:planartriangulation_vtx} implies there exists an orientation $D_T$ of $\overline{T}$ such that $\diam(D_T)= 3= \frac{n_T}{2} +1$ and $ \max \{ d_{D_T}(u_2, w), d_{D_T}(w, u_2)\}\le 2 =\frac{n_T}{2}$. If $\overline{T}\ne K_4$, we apply induction to $\overline{T}$ and get an optimal orientation $D_T$ such that $\diam(D_T)\le \ceil*{\frac{n_T}{2}}$. Therefore, we give an orientation $D_T$ of $\overline{T}$ such that $\diam(D_T)\le \frac{n_T}{2} +1$ and $ \max \{ d_{D_T}(u_2, w), d_{D_T}(w, u_2)\}\le \ceil*{\frac{n_T}{2}}$. Next we orient the two edges incident with $v_3$ by making $O_3$ a directed triangle. Now we orient the edges in $O_i- u_iu_{i+1}$ for $i\in [2]$. Since $O_i$ is maximal outerplanar, $O_i$ has an orientation $D_i$ such that $\diam(D_i)\le \frac{n_i}{2} +1$ and $\max \{ d_{D_i} (u_2, w), d_{D_i} (w, u_2)\} \le \ceil*{\frac{n_i}{2}}$ by Lemma~\ref{lem:outerplanar_vtx}. Observe that $\overline{T}$ and $O_i$ have exactly one common edge for each $i\in [2]$. We can combine $D_i$ with $D_T$ (possibly by reversing the orientation of $D_i$). Therefore we get an orientation $D$ of $G$. 

\begin{figure}[htb]
    \hbox to \hsize{
	\hfil
	\resizebox{4.5cm}{!}{
 
\tikzset{
pattern size/.store in=\mcSize, 
pattern size = 5pt,
pattern thickness/.store in=\mcThickness, 
pattern thickness = 0.3pt,
pattern radius/.store in=\mcRadius, 
pattern radius = 1pt}
\makeatletter
\pgfutil@ifundefined{pgf@pattern@name@_868ikwcts lines}{
\pgfdeclarepatternformonly[\mcThickness,\mcSize]{_868ikwcts}
{\pgfqpoint{0pt}{0pt}}
{\pgfpoint{\mcSize+\mcThickness}{\mcSize+\mcThickness}}
{\pgfpoint{\mcSize}{\mcSize}}
{\pgfsetcolor{\tikz@pattern@color}
\pgfsetlinewidth{\mcThickness}
\pgfpathmoveto{\pgfpointorigin}
\pgfpathlineto{\pgfpoint{\mcSize}{0}}
\pgfusepath{stroke}}}
\makeatother

 
\tikzset{
pattern size/.store in=\mcSize, 
pattern size = 5pt,
pattern thickness/.store in=\mcThickness, 
pattern thickness = 0.3pt,
pattern radius/.store in=\mcRadius, 
pattern radius = 1pt}
\makeatletter
\pgfutil@ifundefined{pgf@pattern@name@_ambav7zjo lines}{
\pgfdeclarepatternformonly[\mcThickness,\mcSize]{_ambav7zjo}
{\pgfqpoint{0pt}{0pt}}
{\pgfpoint{\mcSize+\mcThickness}{\mcSize+\mcThickness}}
{\pgfpoint{\mcSize}{\mcSize}}
{\pgfsetcolor{\tikz@pattern@color}
\pgfsetlinewidth{\mcThickness}
\pgfpathmoveto{\pgfpointorigin}
\pgfpathlineto{\pgfpoint{\mcSize}{0}}
\pgfusepath{stroke}}}
\makeatother
\tikzset{every picture/.style={line width=0.75pt}} 

\begin{tikzpicture}[x=0.75pt,y=0.75pt,yscale=-1,xscale=1]

\draw  [color={rgb, 255:red, 255; green, 255; blue, 255 }  ,draw opacity=1 ][fill={rgb, 255:red, 155; green, 155; blue, 155 }  ,fill opacity=0.3 ] (243.18,61.4) -- (329.37,198.93) -- (156.99,198.93) -- cycle ;
\draw    (242.93,61.4) -- (156.99,198.93) ;
\draw    (328.87,198.93) -- (156.99,198.93) ;
\draw    (242.93,61.4) -- (328.87,198.93) ;
\draw  [draw opacity=0][pattern=_868ikwcts,pattern size=15pt,pattern thickness=0.75pt,pattern radius=0pt, pattern color={rgb, 255:red, 0; green, 0; blue, 0}] (329.37,198.93) .. controls (351.67,185.07) and (352.05,143.4) .. (330.01,105.09) .. controls (307.7,66.29) and (270.84,45.64) .. (247.69,58.95) .. controls (245.36,60.29) and (243.26,61.93) .. (241.38,63.82) -- (288.1,129.2) -- cycle ; \draw   (329.37,198.93) .. controls (351.67,185.07) and (352.05,143.4) .. (330.01,105.09) .. controls (307.7,66.29) and (270.84,45.64) .. (247.69,58.95) .. controls (245.36,60.29) and (243.26,61.93) .. (241.38,63.82) ;  
\draw  [draw opacity=0][pattern=_ambav7zjo,pattern size=15pt,pattern thickness=0.75pt,pattern radius=0pt, pattern color={rgb, 255:red, 0; green, 0; blue, 0}] (159.78,200.13) .. controls (159.74,200.66) and (159.73,201.19) .. (159.73,201.72) .. controls (159.73,225.07) and (197.45,244) .. (243.99,244) .. controls (290.52,244) and (328.25,225.07) .. (328.25,201.72) .. controls (328.25,198.93) and (327.71,196.2) .. (326.68,193.56) -- (243.99,201.72) -- cycle ; \draw   (159.78,200.13) .. controls (159.74,200.66) and (159.73,201.19) .. (159.73,201.72) .. controls (159.73,225.07) and (197.45,244) .. (243.99,244) .. controls (290.52,244) and (328.25,225.07) .. (328.25,201.72) .. controls (328.25,198.93) and (327.71,196.2) .. (326.68,193.56) ;  
\draw  [fill={rgb, 255:red, 0; green, 0; blue, 0 }  ,fill opacity=1 ] (236.91,61.4) .. controls (236.91,64.72) and (239.6,67.42) .. (242.93,67.42) .. controls (246.26,67.42) and (248.95,64.72) .. (248.95,61.4) .. controls (248.95,58.07) and (246.26,55.38) .. (242.93,55.38) .. controls (239.6,55.38) and (236.91,58.07) .. (236.91,61.4) -- cycle ;
\draw  [fill={rgb, 255:red, 0; green, 0; blue, 0 }  ,fill opacity=1 ] (153.76,200.13) .. controls (153.76,203.46) and (156.46,206.16) .. (159.78,206.16) .. controls (163.11,206.16) and (165.81,203.46) .. (165.81,200.13) .. controls (165.81,196.81) and (163.11,194.11) .. (159.78,194.11) .. controls (156.46,194.11) and (153.76,196.81) .. (153.76,200.13) -- cycle ;
\draw  [fill={rgb, 255:red, 0; green, 0; blue, 0 }  ,fill opacity=1 ] (320.65,199.58) .. controls (320.65,202.9) and (323.35,205.6) .. (326.68,205.6) .. controls (330,205.6) and (332.7,202.9) .. (332.7,199.58) .. controls (332.7,196.25) and (330,193.56) .. (326.68,193.56) .. controls (323.35,193.56) and (320.65,196.25) .. (320.65,199.58) -- cycle ;
\draw  [fill={rgb, 255:red, 0; green, 0; blue, 0 }  ,fill opacity=1 ] (126.91,76.4) .. controls (126.91,79.72) and (129.6,82.42) .. (132.93,82.42) .. controls (136.26,82.42) and (138.95,79.72) .. (138.95,76.4) .. controls (138.95,73.07) and (136.26,70.38) .. (132.93,70.38) .. controls (129.6,70.38) and (126.91,73.07) .. (126.91,76.4) -- cycle ;
\draw    (132.93,76.4) -- (159.78,200.13) ;
\draw    (132.93,76.4) -- (242.93,61.4) ;

\draw (164.2,108.48) node [anchor=north west][inner sep=0.75pt]    {$O_{3}$};
\draw (297.49,108.48) node [anchor=north west][inner sep=0.75pt]    {$O_{1}$};
\draw (230.84,217.65) node [anchor=north west][inner sep=0.75pt]    {$O_{2}$};
\draw (223.84,128.51) node [anchor=north west][inner sep=0.75pt]    {$\overline{T}$};
\draw (236.34,36.4) node [anchor=north west][inner sep=0.75pt]    {$u_1$};
\draw (331.37,204.37) node [anchor=north west][inner sep=0.75pt]    {$u_2$};
\draw (144,204.37) node [anchor=north west][inner sep=0.75pt]    {$u_3$};
\draw (107,58.4) node [anchor=north west][inner sep=0.75pt]    {$v_{3}$};

\end{tikzpicture}}%
	\hfil
    }
    \caption{$G$, the case when $|V(O_3)|=3$ and $O_1,O_2$ are maximal outerplanar graphs with order at least $4$.}
    \label{fig:G final case}
    \end{figure}
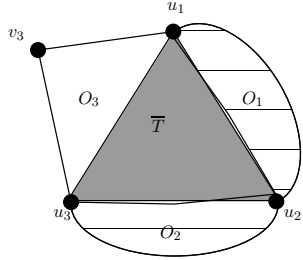

We claim that $\diam(D)\le \ceil*{\frac{n}{2}}$. Let $u,v$ be any two distinct vertices in $G$.

Suppose both $u$ and $v$ are in $V(O_i)$ where $i\in [2]$. We know that $d_D(u,v) \leq d_{D_i}(u,v)\le \frac{n_i}{2}+ 1 \le \frac{n-5}{2}+1\le \ceil*{\frac{n}{2}}$. Similarly, if $u,v\in V(\overline{T})$, then $d_D(u,v) \leq d_{D_T}(u,v)\le \frac{n_T}{2}+ 1 \le \frac{n-5}{2}+1\le \ceil*{\frac{n}{2}}$. Now we assume that one of $u,v$ is in $V(O_i)$ where $i\in [2]$, and one is in $V(\overline{T})$. Without loss of generality, let $u\in V(O_i)$ and $v\in V(\overline{T})$. Then
\begin{align*}
d_D(u,v) &\le d_D(u, u_2)+ d_D(u_2, v) \\ & =d_{D_i}(u,u_2)+d_{D_T}(u_2,v)\\
&\le \ceil*{\frac{n_i}{2}}+\ceil*{\frac{n_T}{2}}\le  \frac{n_i+n_T}{2} +1=\frac{n+3-n_{3-i}}{2}+1 \\& \le \frac{n-1}{2}+1\le \frac{n+1}{2}.
\end{align*}
This implies $d_D(u,v)\le \ceil*{\frac{n}{2}}$. Similarly we can show $d_D(v,u) \le \ceil*{\frac{n}{2}}$; moreover, similarly we can show $d_D(u,v)\le \ceil*{\frac{n}{2}}$ when one of $u,v$ is in $V(O_i)$ and one is in $V(O_{3-i})$ for each $i\in [2]$. Now suppose $v_3\in \{u,v\}$ and we may assume $u=v_3$. If $v\in O_3$, then $d_D(u,v)\leq 2 \leq \ceil{\frac{n}{2}}$. If $v\in V(O_i)$ where $i\in [2]$, then $$d_D(v_3,v)\le d_D(v_3, u_{2i-1})+ d_D(u_{2i-1}, v) \le 2+ \frac{n_i}{2} +1 = \frac{n_i}{2}+3\le \frac{n-5}{2} +3 =\frac{n+1}{2}.$$
Similarly, we can show that $d_D(v, v_3)\le \frac{n+1}{2}$ and that $\max\{ d_D(v_3, v), d_D(v, v_3)\}\le \frac{n+1}{2}$ if $v\in V(\overline{T})$. Hence for any $u,v\in V(G)$, $d_D(u,v)\le \ceil*{\frac{n}{2}}$, i.e., $\diam(D)\le \ceil*{\frac{n}{2}}$. Therefore, $\ordiam(G)\le \diam(D)\le \ceil*{\frac{n}{2}}$.

\vspace{2mm}
Case 3: Suppose $|V(O_i)|\ge 4$ for all $i\in [3]$. Then since $O_i$ is outerplanar and $v_i$ is the only degree-$2$ vertex in $O_i$, it follows that $v_i$ must have a neighbor $v_i'$ in $O_i$ such that $d_G(v_i')=3$ and $v_i'\in V(O_i)\backslash \{u_i, u_{i+1}\}$. Let $H:=G-\{ v_1, v_1', v_2, v_2', v_3, v_3'\}$. Note that $H$ is a near triangulation on $n-6$ vertices and $H$ is not outerplanar. If $H\notin \{ K_4, W_5, G_6^3\}$, we apply induction to $H$ and then $\ordiam(H)\le \ceil*{\frac{n-6}{2}}$. Thus it follows from Lemma~\ref{lem:three_deg2} that $\ordiam(G)\le \ordiam(H)+3\le \ceil*{\frac{n-6}{2}}+3 \le \ceil*{\frac{n}{2}}.$ Now we assume that $H\in \{ K_4, W_5, G_6^3\}$. Observe that $H\ne W_5$ as $\overline{T}\subseteq H$. Since $H\in \{ K_4, G_6^3\}$, there exists a unique choice of the embedding of $\overline{T}$ in $H$. It follows that there are at least two $O_i$ such that $O_i=K_4^-$ and we may assume $O_1, O_3$ are $K_4^-$. Let $H'=G -\{v_1, v_1', v_3, v_3'\}$. Observe that $u_1\in V(O_1)\cap V(O_3)$ and $u_1$ is contained in the outer cycle of $H'$, say $C'$, and $d_{H'}(u_1)=3$ and $u_1$ has a neighbor in $V(H')\setminus V(C')$. Recall that $H\in \{K_4, G_6^3\}$. If $H=K_4$, we know that $n=10$ and $H'=G_6^3$. By Lemma~\ref{lem:planartriangulation_vtx}, there exists an orientation $D'$ of $H'$ such that $\diam(D')=\ordiam(H')=4= \frac{n}{2}-1$ and $\max \{d_D(u_1, w), d_D(w, u_1) \}\le 3=\frac{n}{2}-2$ for all $w\in V(H')$. If $H=G_6^3$, $n=12$ and $H'$ is a near triangulation on $8$ vertices such that $H'$ is not outerplanar. By induction, $\ordiam(H')\le 4 = \frac{n}{2}-2$. Therefore, $H'$ has an orientation $D'$ such that $\diam(D')\le \frac{n}{2}-1$ and $\max\{ d_{D'}(u_1, w),  d_{D'}(w, u_1)\}\le \frac{n}{2}-2$.
For $i=1,3$, $O_i=K_4^-$ and $O_i$ has an orientation $D_i$ such that $\diam(D_i)=3$ and by Lemma~\ref{lem:outerplanar_vtx}, $\max \{ d_{D_i}(u_1, w),d_{D_i}( w, u_1) \}\le 2$ for every $w\in V(O_i)$. Combine $D_1, D_3, D'$ and we get an orientation $D$ of $H$. It is not hard to check $\diam(D)\le \frac{n}{2}$. Therefore, $\ordiam(G)\le \diam(D)\le \ceil*{\frac{n}{2}}$. This completes the proof.


\begin{thebibliography}{}

\bibitem{BBRV2021}
J. Babu, D. Benson, D. Rajendraprasad, and S. N. Vaka, An improvement to Chv\'atal and Thomassen’s upper bound for oriented diameter, \textit{Discrete Appl. Math.}, \textbf{304} (2021), 432--440.

\bibitem{Bau-Dankelmann2015}
S. Bau and P. Dankelmann, Diameter of orientations of graphs with given minimum degree, \textit{European J. Combin.}, \textbf{49} (2015), 126–133.


\bibitem{Chen-Chang2021}
B. Chen and A. Chang, Diameter three orientability of bipartite Graphs, \textit{Electron. J. Combin.}, \textbf{28(2)} (2021), P2.25.

\bibitem{CGT1985}
F. R. K. Chung, M. R. Garey, and R. E. Tarjan, Strongly connected orientations of mixed multigraphs, \textit{Networks}, \textbf{15(4)} (1985), 477--484.

\bibitem{Chvatal-Thomassen1978}
V. Chv\'atal and C. Thomassen, Distances in orientations of graphs, \textit{J. Combin. Theory Ser. B}, \textbf{24(1)} (1978), 61--75.

\bibitem{Cochran2023+}
G. Cochran, Large girth and small oriented diameter graphs, 
arXiv.2201.07618.

\bibitem{CCDS2021}
G. Cochran, E. Czabarka, P. Dankelmann, and L. Sz\'ekely, A size condition for diameter two orientable graphs, \textit{Graphs Combin.}, \textbf{37(2)} (2021), 527–544.

\bibitem{CDS2019}
E. Czabarka, P. Dankelmann, and L. Sz\'ekely, A degree condition for diameter two orientability of graphs, \textit{Discrete Math.}, \textbf{342(4)} (2019), 1063–1065.

\bibitem{DGS2018}
P. Dankelmann, Y. Guo, and M. Surmacs, Oriented diameter of graphs with given maximum degree, \textit{J. Graph Theory}, \textbf{88(1)} (2018), 5–17.

\bibitem{EPPT1989}
P. Erd\H{o}s, J. Pach, R. Pollack, and Z. Tuza, Radius, diameter, and minimum degree, \textit{J. Combin. Theory Ser. B}, \textbf{47(1)} (1989), 73--79.

\bibitem{FMPR2004}
F. V. Fomin, M. N. Matamala, E. Prisner, and I. Rapaport, AT-free graphs: linear bounds
for the oriented diameter, \textit{Discrete Appl. Math.}, \textbf{141(1)} (2004), 135--148.

\bibitem{FMR2004}
F. V. Fomin, M. Matamala, and I. Rapaport, Complexity of approximating the oriented diameter of chordal graphs, \textit{J. Graph Theory}, \textbf{45(4)} (2004), 255-–269.

\bibitem{Gutin1994}
G. Gutin, Minimizing and maximizing the diameter in orientations of graphs, \textit{Graphs Combin.}, \textbf{10(2-4)} (1994), 225–230.

\bibitem{GKTY2002}
G. Gutin, K. M. Koh, E. Tay, and A. Yeo, Almost minimum diameter orientations of semicomplete multipartite and extended digraphs, \textit{Graphs Combin.}, \textbf{18(3)} (2002), 499–506.

\bibitem{Gutin-Yeo2002}
G. Gutin and A. Yeo, Orientations of digraphs almost preserving diameter, \textit{Discrete Appl. Math.}, \textbf{121(1)} (2002), 129–138.



\bibitem{Huang-Ye2007}
J. Huang and D. Ye, Sharp bounds for the oriented diameters of interval graphs and 2-connected proper interval graphs, \textit{Computational Science – ICCS 2007 (Series Title: Lecture Notes in Computer Science.)}, \textbf{4489} (2007), 353--361. 

\bibitem{Koh-Ng2005}
K. M. Koh and K. L. Ng, The orientation number of two complete graphs with linkages, \textit{Discrete Math.}, \textbf{295(1)} (2005), 91–106.

\bibitem{Koh-Tay2022}
K. M. Koh and E. G. Tay, Optimal orientations of graphs and digraphs: a survey, \textit{§Graphs Combin.}, \textbf{18(4)} (2022), 745–756.



\bibitem{KRS2022}
K. Kumar, D. Rajendraprasad, and K. Sudeep, Oriented diameter of star graphs, \textit{Discrete Appl. Math.}, \textbf{319} (2022), 362--371.

\bibitem{KLW2010}
P. K. Kwok, Q. Liu, and D. B. West, Oriented diameter of graphs with diameter $3$, \textit{J. Combin. Theory Ser. B}, \textbf{100(3)} (2010), 265--274.

\bibitem{Laetsch-Kurz2012}
M. Laetsch and S. Kurz, Bounds for the minimum oriented diameter. \textit{Discrete Math. Theor. Comput. Sci.}, \textbf{14(1)} (2012), 109--142.

\bibitem{Lakshmi2011}
R. Lakshmi, Optimal orientation of the tensor product of a small diameter graph and a complete graph, \textit{Australas. J. Combin}, \textbf{50} (2011), 165–169.


\bibitem{Lakshmi-Paulraja2007}
R. Lakshmi and P. Paulraja, On optimal orientations of tensor product of complete graphs, \textit{Ars Combinatoria}, \textbf{82} (2007), 337–352.

\bibitem{Lakshmi-Paulraja2009}
R. Lakshmi and P. Paulraja, On optimal orientations of tensor product of graphs and circulant graphs, \textit{Ars Combinatoria}, \textbf{92} (2009), 271–288.

\bibitem{MPR2023}
D. Mondal, N. Parthiban, I. Rajasingh, Bounds for the Oriented Diameter of Planar Triangulations, \textit{Frontiers of Algorithmic Wisdom, IJTCS-FAW 2022}, (2023), 192-205.

\bibitem{Plesnik1985}
J. Plesn\'ik, Remarks on the diameters of orientations of graphs, \textit{Acta Math. Univ. Comenian}, \textbf{36(3)} (1985), 225–236.

\bibitem{Robbins1939}
H. E. Robbins, A theorem on graphs, with an application to a problem of traffic control, \textit{Amer. Math. Monthly}, \textbf{46(5)} (1939), 281--283.

\bibitem{Surmacs2017}
M. Surmacs, Improved bound on the oriented diameter of graphs with given minimum degree, \textit{European J. Combin.}, \textbf{59} (2017), 187--191.

\bibitem{Soltes1986}
 L. \v{S}olt\'es, Orientations of graphs minimizing the radius or the diameter, \textit{Math. Slovaca}, \textbf{36(3)} (1986), 289–296.

\bibitem{WCDGSV2021}
X. Wang, Y. Chen, P. Dankelmann, Y. Guo, M. Surmacs, and L. Volkmann. Oriented diameter of maximal outerplanar graphs, \textit{J. Graph Theory}, \textbf{98(3)} (2021), 426–444.



\end{thebibliography}
\end{document}